\documentclass[10pt, twoside]{amsart}

\usepackage[utf8]{inputenc}
\usepackage{hyperref}

\usepackage{graphicx}              
\usepackage{amsmath, bm}               
\usepackage{amsfonts}              
\usepackage{amsthm}                
\usepackage{amssymb}
\usepackage{mathtools}             

\usepackage{enumitem} 

\usepackage[all]{xy}
\usepackage{amscd}

\allowdisplaybreaks
\usepackage{color}
\usepackage{fp}

\newcommand{\Z}{\mathbb{Z}}

\newcommand{\C}{\mathbb{C}}
\newcommand{\R}{\mathbb{R}}

\newcommand{\Hi}{\mathcal{H}}
\newcommand{\Li}{\mathcal{L}}
\newcommand{\X}{\mathsf{X}}
\newcommand{\Y}{\mathsf{Y}}
\newcommand{\W}{\mathsf{W}}
\newcommand{\V}{\mathsf{V}}
\newcommand{\Zm}{\mathsf{Z}}
\newcommand{\Um}{\mathsf{U}}
\newcommand{\cj}{\overline}
\newcommand{\op}{\mathrm}

\newcommand{\lpp}[3]{( #1 \mid #2 )_#3}
\DeclarePairedDelimiterX{\normb}[1]{\lVert}{\rVert}{#1}

\numberwithin{equation}{section}
\newtheorem{theorem}{Theorem}[section]
\newtheorem{lem}[theorem]{Lemma}
\newtheorem{tpr}[theorem]{Proposition}
\newtheorem{cor}[theorem]{Corollary}
\renewcommand{\qedsymbol}{$\blacksquare$}
\let\origproofname\proofname
\renewcommand{\proofname}{\upshape\textbf{\origproofname}}

\theoremstyle{definition}
\newtheorem{ex}[theorem]{Example}
\newtheorem{defi}[theorem]{Definition}
\newtheorem{rmk}[theorem]{Remark}
\newtheorem{nota}[theorem]{Notation}

\newtheorem{quest}[theorem]{Question}

\usepackage{indentfirst}

\begin{document}

\title{$L^p$-modules and $L^p$-correspondences}

\author{Alonso Delfín}

\date{\today}

\address{Department of Mathematics, University of Colorado, Boulder CO 80309-0395, USA}

\email[]{alonso.delfin@colorado.edu}

\thanks{\textsc{Department of Mathematics, University of Colorado, Boulder CO 80309-0395, USA}
}

\maketitle

\begin{abstract}
We introduce an $L^p$-operator algebraic analogue of Hilbert C*-modules. We present the theory of concrete $L^p$-modules, their morphisms, and basic constructions including countable direct sums and tensor products.
We then define $L^p$-correspondences and the interior tensor product of these. 
\end{abstract}

\tableofcontents

\section{Introduction}

Hilbert C*-modules have been widely used as a tool to study C*-algebras. 
For instance, these modules are one of the main ingredients used to study 
Morita equivalence and KK-theory of C*-algebras. Furthermore, 
Hilbert modules are used to define C*-correspondences, another C*-theoretic 
tool that appears constantly in the current literature. One of the main uses 
of such correspondences is that they give rise to the so called 
Cuntz-Pimsner algebras introduced by M. Pimsner in \cite{pim1997} 
and later refined by T. Katsura in a series of papers (see for instance \cite{kat2004} and \cite{kat2007}). 
The class of Cuntz-Pimsner algebras contains several known 
examples of C*-algebras such as the classical Cuntz algebras, 
Cuntz-Krieger algebras, crossed products by $\Z$, and topological 
graph C*-algebras. 

In recent years, N. C. Phillips revived the interest in 
algebras of operators acting on $L^p$-spaces 
(originally studied by C. Herz in \cite{CSH71}) 
by defining $p$-analogues of the Cuntz algebras 
and crossed products (see \cite{ncp2013CP} and \cite{ncp2012AC}). 
Since then, the study of these algebras has gained 
significant interest and several authors have contributed 
to the expansion of this relatively new theory (see for example \cite{blph2020, BDW2024, Chung2024, Daw10, GarThi16, GarThi20, ncpmgv2017AF}). 
Many objects in the widely studied Cuntz-Pimsner-Katsura class are now known to have a $p$-analogue, 
which raises a natural question on whether such class has a $p$-counterpart. This is, of course, a very general and difficult question 
given that some ``C*-closed'' constructions, such as taking quotients \cite{GarThi16,blph2020} or multiplier algebras \cite{BDW2024}, are not generally well behaved in the $L^p$-operator algebra setting. However, a starting point is to attempt an analogue theory 
to the one of Hilbert C*-modules and C*-correspondences in the $L^p$-setting. 
This is exactly what this paper achieves, at least at the concrete level. 

In this work we focus mainly on developing the theory of $L^p$-modules, 
their morphisms, and basic constructions such as countable direct sums and tensor products.
We also give, when possible, 
instances in which well known facts for Hilbert C*-modules extend to the $L^p$-case. All this allows us to define $L^p$-correspondences in a natural way. A follow-up project, currently in preparation, is to use some of the  results in this paper to define an $L^p$-analogue of the Cuntz-Pimsner class that includes the $p$-version 
of the Cuntz algebras and crossed products of $L^p$-operator algebras by $\Z$. 
Even though in this paper we do not attempt more applications
of $L^p$-modules, we do observe that  $L^p$-modules are in fact
 Banach pairs as defined by V. Lafforgue and that our notion of morphisms agrees with that of linear operators of Banach pairs (see \cite{Laf02,paravicini_2009} for instance). 
Thus, we believe this work can be applied to the general study of 
KK-theory and Morita equivalence of $L^p$-operator algebras (see \cite{Chung2024} for 
recent work along these lines). 

\textbf{Structure of the paper and main results:} 
Section \ref{s3} contains all the notational conventions and the necessary background and 
    references for $L^p$-operator algebras. In Sections \ref{s4} and \ref{s5}, we take advantage of the main results in \cite{Delfin_2022},
     where C*-correspondences are concretely represented on pairs of Hilbert spaces, to naturally define 
      $L^p$-modules and $L^p$-correspondences as a generalization of the C*-case. The main idea is that we are replacing 
      Hilbert spaces with $L^p$-spaces. Indeed, roughly speaking, our Definition \ref{mainD} for an
      $L^p$-module $(\Y, \X)$ comes by looking at the conditions satisfied by the pair $(\pi_\X(\X)^*, \pi_\X(\X))$
      in Definition 3.7 from \cite{Delfin_2022}. A consequence of this definition is that any 
      $L^p$-module $(\Y, \X)$ over an $L^p$-operator algebra $A$
      comes equipped with a pairing $\Y \times \X \to A$. Those $L^p$-modules for which their norm can be 
      recovered using such pairing are called \textit{C*-like $L^p$-modules}, so that any Hilbert module over a C*-algebra 
      $A$ is actually a C*-like $L^2$-module. We then further develop the general theory of $L^p$-modules by 
      presenting several examples and classical constructions such as their finite direct sums, countable direct sums, external tensor products, 
      and finally the notion of $L^p$-module morphisms (Definition \ref{GenMorph}) and the 
      $L^p$-module compact morphisms (Definition \ref{GenKMorph}). One advantage of these concrete definitions is that the algebra of morphisms from an $L^p$-module to itself,  denoted by 
      $\Li_A( \Y,\X )$ (see Equation \eqref{L_A(X,Y)}), comes naturally equipped with an $L^p$-operator algebra structure and is in fact a generalization of the C*-algebra of adjointable maps on a Hilbert module. 
       Similarly, in Equation (\ref{KK_A(X,Y)}), we get the ideal $\mathcal{K}_A(\Y,\X)$, which is a generalization
       of compact-module maps in the Hilbert module setting.
       
        Our main results 
       can be summarized as follows:
       \begin{enumerate}
       \item Theorem \ref{CDsumisLpmod} in which we show that our notion of countable direct sum 
       of $L^p$-modules agrees with the classical Hilbert module one. 
       \item Proposition \ref{stdL^pMod} 
       in which we show that, just as in the C*-case, tensoring the $L^p$-module $(\ell^q, \ell^p)$ over $\C$ against any $L^p$-module over $A$ corresponds simply to the countable direct sum of the module, 
       \[
       (\ell^q,\ell^p) \otimes_p (\Y, \X) = \bigoplus_{j = 1}^\infty( \Y, \X).
       \]
       \item Proposition \ref{Kasp} in which we show that the standard $L^p$-module 
       of a nondegenerate approximately unital $L^p$-operator algebra $A$ satisfies the $p$-version of
        Kasparov's theorem (Theorem 15.2.12 in \cite{wegge-olsen_2004}):
        \[
      \mathcal{K}_A(   (\ell^q, \ell^p) \otimes_p (A, A)   )  \cong \mathcal{K}(\ell^p) \otimes_p A, \]
      \[ \Li_A(  (\ell^q, \ell^p)  \otimes_p (A, A)  ) \cong M(  \mathcal{K}(\ell^p) \otimes_p A ).
        \]
\end{enumerate}        
       
       The definition of ``adjointable'' maps from an $L^p$-module to itself naturally
       gives rise to the concept of $L^p$-correspondence (Definition \ref{mainLpcorrresD}). 
       Since representations of C*-correspondences on pairs of Hilbert spaces are, in some sense, 
       well behaved with respect to the interior tensor product (Theorem 4.14 in \cite{Delfin_2022}),
       we deduce from there an analogous interior tensor product construction for the $L^p$-case
       (see Definition \ref{TensorLpCorres}). Having all these tools at our disposition while working 
       with $L^p$-correspondence provides evidence that we should be able to carry an analogue of the usual 
        Fock representations and the Fock space construction  (see Definitions 4.1 and 4.2 in \cite{kat2004})
for $L^p$-correspondences. This is currently being carried as a separate project that will use some 
of the results given in this paper. 
       
\section*{Acknowledgments} Part of this work comes from the author's doctoral dissertation \cite{Del2023}. 
       The author would like to thank his advisor, N. Christopher Phillips, for the advice given 
       during graduate school, particularly for proposing the problem of finding 
        analogues of C*-correspondences in the $L^p$-setting and for 
        carefully reading earlier versions 
       of this document. 

\section{Preliminaries}\label{s3}

If $E, F$ are Banach spaces, we write $\Li(E,F)$ for the Banach space of bounded linear maps  
from $E$ to $F$, equipped with the usual operator norm. As usual we write $\Li(E)$ for $\Li(E,E)$. 

Recall that a Banach algebra $A$ is said to have 
a \emph{contractive approximate identity} (c.a.i.\ from now on) if 
there is a net $(e_\lambda)_{\lambda \in \Lambda}$ in $A$
such that $\| e_\lambda \| \leq 1$ for all $\lambda \in \Lambda$
and for all $a \in A$, 
\[
\lim_{\lambda \in \Lambda} \| ae_\lambda -a \|  =\lim_{\lambda \in \Lambda} \| e_\lambda a -a \| = 0.
\]

\begin{defi}
Let $A$ be a Banach algebra and $E$ a Banach space. A
\emph{representation of $A$ on $E$} is a continuous homomorphism $\pi \colon A \to \Li(E)$.
\begin{enumerate}
\item We say that $\pi$ is \emph{contractive} if $\|\pi(a)\| \leq  \|a\|$ for all $a \in A$.
\item  We say that $\pi$ is \emph{isometric}  if $\|\pi(a)\| =  \|a\|$ for all $a \in A$.
\item  We say that $\pi$ is nondegenerate if
\[
\pi(A)E = \op{span}(\{\pi(a)\xi \colon a \in A \text{ and } \xi \in E\}),
\]
is dense in $E$, and we say that $A$ is \emph{nondegenerately representable} if it has a
nondegenerate isometric representation.
\end{enumerate}
\end{defi}

\subsection{$L^p$-operator algebras.}

If $(\Omega, \mathfrak{M}, \mu)$ is a measure space, 
we define $L^0(\Omega, \mathfrak{M}, \mu)$ to be
the space of complex valued measurable functions modulo functions 
that vanish a.e $[\mu]$. For $p \in [1,\infty]$ we have the classical 
$p$-norms
\[
\| \xi \|_p = \begin{cases}
\left( \int_\Omega |\xi|^p d\mu\right)^{1/p} & \text{ if } p \in [1,\infty)\\
\op{ess{ \ }sup}(|\xi|)  & \text{ if } p = \infty
\end{cases}.
\]
For any $p \in [1,\infty]$ we let $L^p(\Omega, \mathfrak{M}, \mu) = \{ \xi \in L^0(\Omega, \mathfrak{M}, \mu) \colon \| \xi \|_p < \infty\}$.
For $p \in [1,\infty]\cup\{0\}$, most times we
write $L^p(\Omega, \mu)$ or simply $L^p(\mu)$ for $L^p(\Omega, \mathfrak{M}, \mu)$. 
Also, if $\nu_I$ is counting measure on a set $I$, we write $\ell^p(I)$ instead of
$L^p(I, 2^I, \nu_I)$. In particular, when $d \in\Z_{\geq 1}$, 
we simply write $\ell^p_d$ for $\ell^p(\{1, \ldots, d\})$ and 
we also often write $\ell^p$ instead of $\ell^p(\Z_{\geq 1})$. 

Further, if $E$ is any Banach space, we denote by $L^0(\Omega, \mu; E)$ 
the vector space of measurable functions $\Omega \to E$ modulo functions 
that vanish a.e $[\mu]$. For any $p \in [1,\infty]$, the space of $p$-Bochner integrable functions
is defined as
\[
L^p(\Omega,\mu; E) = \{ g \in L^0(\Omega, \mu; E) \colon \omega \mapsto \|g(\omega)\|_E \in L^p(\Omega, \mu)\}.
\]

\begin{defi}\label{Lp}
Let $p \in [1, \infty)$. A Banach algebra $A$ is an 
\textit{$L^p$-operator algebra} if there is a measure space
$(\Omega, \mathfrak{M}, \mu)$ and an isometric representation of $A$
on $L^p(\mu)$.  
\end{defi}

\subsection{ Spatial Tensor Product}

For $p \in [1, \infty)$, there is a Banach space tensor product, called the 
\textit{spatial tensor product} and denoted by $\otimes_p$. 
This tensor product is defined when one of the factors is an $L^p$-space 
and the other an arbitrary Banach space. 
We describe below only the properties of $\otimes_p$ we will need, 
and refer the reader to Section 7 of \cite{defflor1993} for complete details on this 
tensor product.  

If $(\Omega_0, \mathfrak{M}_0, \mu_0)$ is a measure space and $E$ is a Banach space, 
then there is an isometric isomorphism 
\[
L^p(\mu_0) \otimes_p E \cong L^p(\Omega_0,\mu_0; E),
\]
such that for any $\xi \in L^p(\mu_0)$ and $\eta \in E$, the elementary tensor $\xi \otimes \eta $ 
is sent to the function $\omega \mapsto \xi(\omega)\eta$. Furthermore, if $(\Omega_1,  \mathfrak{M}_1, \mu_1)$
is another measure space and $E=L^p(\mu_1)$, then there is an isometric isomorphism
\[
L^p(\mu_0) \otimes_p L^p(\mu_1) \cong L^{p}(\Omega_0 \times \Omega_1, \mu_0 \times \mu_1),
\] 
sending $\xi \otimes \eta$ to the function 
$(\omega_0,\omega_1) \mapsto \xi(\omega_0)\eta(\omega_1)$ for every $\xi \in L^p(\mu_0)$
and $\eta \in L^p(\mu_1)$. 
We describe its main properties below. 
The following is Theorem 2.16 in \cite{ncp2012AC}, 
except that we have removed the  the $\sigma$-finiteness assumption 
as in the proof in Theorem 1.1 in \cite{figiel1984}.
\begin{enumerate}
\item Under the identification above, 
$\op{span}\{ \xi \otimes \eta  \colon \xi \in L^p(\mu_0), \eta \in L^p(\mu_1)\}$ 
is a dense subset of $L^{p}(\Omega_0 \times \Omega_1, \mu_0 \times \mu_1)$. 
\item $\| \xi \otimes \eta \|_p = \| \xi\|_p\|\eta\|_p$  for every 
$\xi \in L^p(\mu_0)$ and $\eta \in L^p(\mu_1)$.
\item Suppose that for $j \in \{0,1\}$ we have measure spaces 
$(\Omega_j, \mathfrak{M}_j, \mu_j)$, $(\Lambda_j, \mathfrak{N}_j, \nu_j)$,
$a \in \Li(L^p(\mu_0), L^p(\nu_0))$
and $b \in \Li(L^p(\mu_1), L^p(\nu_1))$. 
Then there is a unique map 
$a\otimes b \in \Li(L^p(\mu_0 \times \mu_1),
 L^p( \nu_0 \times \nu_1))$ 
such that 
\[
(a \otimes b)(\xi \otimes \eta)=a\xi \otimes b\eta
\]
for every $\xi \in L^p(\mu_0)$ and $\eta \in L^p(\mu_1)$. 
Further, $\| a\otimes b\|=\|a\|\|b\|$.
\item The tensor product of operators defined in (3) is associative, bilinear, 
and satisfies (when the domains are appropriate) 
$(a_1 \otimes b_1)(a_2 \otimes b_2) = a_1 a_2 \otimes b_1b_2$.
\end{enumerate}

\begin{defi}\label{LpT_P}
Let $p \in [1, \infty)$ and let 
 $A \subseteq \Li(L^p(\mu))$ and $B \subseteq \Li(L^p(\nu))$ be $L^p$-operator algebras. 
 We define $
  A \otimes_p B$
   to be the closed linear span, in $\mathcal{L}\big(L^p(\mu \times \nu) \big)$, of all $a \otimes b$ for $a \in A$ and $b \in B$. 
\end{defi}

\begin{rmk}
Definition \ref{LpT_P} provides only a concrete tensor product of $L^p$-operator algebras. 
Different representations for $A$ and $B$ on $L^p$-spaces can yield a different tensor product 
as shown below Example 1.15 in \cite{ncp2013CP}. This issue appears even when $p=2$, in the nonselfadjoint case, but will not happen for C*-algebras. This is fixed in 
\cite[Definition 7.2]{ChoiGardellaThiel24} where the general theory of $L^p$-operator algebras is introduced. 
To be more precise,  let $\op{Rep}_p(A)$ denote all the contractive nondegenerate representations of $A$ on $L^p$-spaces. Then, for any two $L^p$-operator algebras $A$ and $B$, the \textit{spatial tensor product} $A \otimes_{\op{sp}}B$ is defined as the completion of $A \otimes B$
under the norm 
\[
A \otimes B \ni t \mapsto \| t \|_{\op{sp}} = \sup\{ \| (\pi_A \otimes \pi_B)(t) \colon \pi_A \in \op{Rep}_p(A), \pi_B \in \op{Rep}_p(B)\}.
\] 
By construction the identity map extends to a contraction $A \otimes_{\op{sp}}B \to A \otimes_p B$ with dense range. Given the concrete nature of this paper, in which most of our definitions below are a priori dependent on the concrete representation $A \subseteq \Li(L^p(\mu))$, we only work with the tensor product in Definition \ref{LpT_P}.
\end{rmk}

\section{$L^p$-modules over $L^p$-operator algebras}\label{s4}

In this section we initiate the study of a type of modules over $L^p$-operator 
algebras that generalizes Hilbert modules over C*-algebras. The definitions 
here are motivated by the theory of concrete C*-modules (see Section 3 \cite{Murph97}). 

\subsection{$L^p$-modules and C*-like $L^p$-modules}

For our main definition, 
it is worth revisiting Example 2.1 from \cite{Delfin_2022}. Recall
that if $A \subseteq \Li(\Hi_0)$ is a concrete C*-algebra, then 
any closed subspace $\X \subseteq \Li(\Hi_0, \Hi_1)$ satisfying
\begin{enumerate}
\item $xa \in \X$ for all $x \in \X$, $a \in A$, 
\item $x_1^*x_2 \in A$ for all $x_1,x_2 \in \X$,
\end{enumerate}
 is a (concrete) right Hilbert $A$-module. Furthermore, 
observe that the adjoint space $\X^* = \{ x^* \colon x \in \X\}$ is 
 a closed subspace of $\Li(\Hi_1, \Hi_0)$ 
 satisfying 
\begin{enumerate}
\setcounter{enumi}{2}
\item $ay \in \X^*$ for all  $a \in A$, $y \in \X^*$.
\end{enumerate}
Finally, by standard Hilbert module arguments we also know
that the norm of an element $x_0$ in any 
right Hilbert $A$-module $\Y$ agrees with the operator norm of the map $x \mapsto \langle x_0, x\rangle_A$
which is in $\Li_A(\X, A)$ with adjoint 
given by $a \mapsto x_0a$.
For concrete Hilbert modules, this is equivalent to asking that for any $x_0 \in \X$ and $y_0 \in \X^*$
 \begin{enumerate}
\setcounter{enumi}{3}
\item $\| x_0 \| = \sup_{y \in \X^*, \|y\|=1} \| yx_0\|$ and $\| y_0\| = \sup_{x \in \X, \|x\|=1} \| y_0x\|$
\end{enumerate}
Our main definition of $L^p$-modules is motivated by the behavior 
we just described for the pair $(\X^*, \X)$.

\begin{defi}\label{mainD}
Let $(\Omega_0, \mathfrak{M}_0, \mu_0)$ and $(\Omega_1, \mathfrak{M}_1, \mu_1)$
be measure spaces, let $p \in [1, \infty)$, 
and let $A \subseteq \Li(L^p(\mu_0))$ be an $L^p$-operator 
algebra. An \emph{$L^p$-module} over $A$ is a pair $(\Y, \X)$, 
where $\Y \subseteq  \Li(L^p(\mu_1), L^p(\mu_0))$ and 
$\X \subseteq \Li(L^p(\mu_0), L^p(\mu_1))$ are closed subspaces  
satisfying
\begin{enumerate}
\item $xa \in \X$ for all $x \in \X$, $a \in A$, \label{lpm5}
\item $ay \in \Y$ for all $y \in \Y$, $a \in A$, \label{lpm5.1}
\item $yx \in A$ for all $y \in \Y$, $x \in \X$. \label{lpm5.2}
\end{enumerate}
If in addition for every $x_0 \in \X$ and $y_0 \in \Y$ we have 
 \begin{enumerate}
\setcounter{enumi}{3}
\item  $\| x_0 \| = \sup_{y \in \Y, \|y\|=1} \|yx_0\|$ and $\| y_0 \| = \sup_{x \in \X, \|x\|=1} \|y_0x\|$, \label{lpm6}
 \end{enumerate}
then we say that $(\Y, \X)$ is a \emph{C*-like $L^p$-module}. 
 \end{defi}
 
 \begin{rmk}\label{BanPA1}
 Observe that Conditions \eqref{lpm5} and \eqref{lpm5.1} in Definition \ref{mainD}
 give that $\X$ is a right Banach $A$-module and that $\Y$ is a left Banach $A$-module. 
 Together with Condition \eqref{lpm5.2} we see that the pair $(\Y, \X)$ is a Banach $A$-pair 
 in the sense of Lafforgue (see \cite[Définition 1.1.3]{Laf02}, \cite[Section 1]{paravicini_2009}). 
 \end{rmk}
 
 \begin{nota}\label{lppair}
If $(\Y, \X)$ is an $L^p$-module over $A$, it comes 
naturally equipped with a pairing $\Y \times \X \to A$ via 
$(y, x) \mapsto yx$. It will be convenient
to sometimes denote the operator $yx: L^p(\mu_0) \to L^p(\mu_0)$ by $( y \mid x )_A$.
 \end{nota}
 
We now present various examples of $L^p$-modules. 

\begin{ex}
Let $A$ be a C*-algebra and let $\X$ be any right Hilbert $A$ module. 
If $(\pi_A, \pi_\X)$ is an isometric representation of $\X$ on 
a pair of Hilbert spaces $(\Hi_0, \Hi_1)$ as in Definition 3.7 in \cite{Delfin_2022}, then $(\pi_\X(\X)^*, \pi_\X(X))$ is a C*-like $L^2$-module 
over the C*- algebra $\pi_A(A)$.  
\end{ex}

\begin{ex}\label{A_A}
Let $p \in [1, \infty)$, let $(\Omega, \mathfrak{M}, \mu)$
be a measure space, and let $A \subseteq \Li(L^p(\mu))$
be an $L^p$-operator algebra. Then $(A, A)$ is trivially 
an $L^p$-module over $A$.
However, $(A,A)$ is not always C*-like, as Condition 
(\ref{lpm6}) from Definition \ref{mainD} does not generally hold when $A$ is non-unital. 
Indeed, if 
\[
A = \Biggl\{ \Biggl( \begin{matrix}
0 & z \\
0 & 0
\end{matrix} \Biggr) \colon z \in \C \Biggr\} \subset M_2^p(\C)=\Li(\ell^p_2),
\]
then 
\[
\Biggl\|\Biggl( \begin{matrix}
0 & 1 \\
0 & 0
\end{matrix} \Biggr)\Biggr\| = 1 > 0 = \sup_{|z|=1} \Biggl\|\Biggl( \begin{matrix}
0 & z \\
0 & 0
\end{matrix} \Biggr)\Biggl( \begin{matrix}
0 & 1\\
0 & 0
\end{matrix} \Biggr)\Biggr\|.
\]
Nevertheless, if $A$ has a c.a.i.,
then it is immediate to see that $(A,A)$ is C*-like.
\end{ex}

\begin{ex}\label{LpLqMOD}
Let $(\Omega, \mathfrak{M}, \mu)$ be a measure space, let $p \in (1, \infty)$, 
and consider the $L^p$-operator algebra $A=\Li(\ell_1^p)$. Observe that 
$A$ can be identified with $\C$ via $a \mapsto a(1)$ and that $\|a\|=|a(1)|$ for any $a \in A$, 
whence the identification is isometric. Now let $\X = L^p(\mu)$, which we 
isometrically identify with $\Li(\ell_1^p, L^p(\mu))$ via $\xi \mapsto (z \mapsto z\xi)$ for any 
$\xi \in L^p(\mu)$ and $z \in \ell_1^p$.
Similarly, if $q$ is the Hölder conjugate of $p$, then $\Y=L^q(\mu)$
is isometrically identified with $\Li(L^p(\mu), \ell_1^p)$ via the usual dual 
pairing $\eta \mapsto (\xi \mapsto \langle \eta, \xi \rangle=\int_\Omega \eta\xi d\mu)$ 
for $\eta \in L^q(\mu)$ and $\xi \in L^p(\mu)$.  
Under these identifications, we claim that $(\Y, \X)$ is a C*-like $L^p$-module 
over $A$. Clearly $\X$ and $\Y$ are closed subsets of $\Li(\ell_1^p, L^p(\mu))$ and $\Li(L^p(\mu), \ell_1^p)$ respectively. 
We check that Conditions \eqref{lpm5}-\eqref{lpm6} from Definition \ref{mainD} hold. 
Let $\xi \in \X$ and let $a \in A$. Then the composition $\xi a: \ell_1^p \to L^p(\mu)$
is clearly a bounded linear map, proving Condition \eqref{lpm5}.
Similarly, for $a\in A$ and $\eta \in Y$, 
we note that the composition $a\eta \colon L^p(\mu) \to \ell_1^p$ is a bounded linear map and 
therefore Condition \eqref{lpm5.1} is done. 
If $\eta \in \Y$ and $\xi \in \X$, 
the composition $\lpp{\eta}{\xi}{A} \colon \ell_1^p \to \ell^p_1$ agrees with $\langle \eta, \xi \rangle$ as an 
element of $A$, 
so Condition \eqref{lpm5.2} follows.
Finally, Hölder duality gives that for any $\xi_0 \in \X$ and $\eta_0 \in Y$
$\| \xi_0 \|_p = \sup_{\| \eta \|_q = 1 } |\langle \eta, \xi_0 \rangle|$ and 
$\| \eta_0 \|_q = \sup_{\| \xi \|_p = 1 } |\langle \eta_0, \xi \rangle|$, 
so Condition \eqref{lpm6} also follows. 
\end{ex}

\begin{ex}\label{lplqMOD}
Let $d \in \Z_{\geq 1}$, let $p \in [1, \infty)$,
and let $q$ be the Hölder conjugate of $p$. 
As particular instance of Example \ref{LpLqMOD}, we see that 
 $(\ell_d^q, \ell_d^p)$ is a C*-like $L^p$-module 
 over $\C$. Notice that
we are now able to include $p=1$ because 
the dual of $\ell_d^1$ is $\ell_d^\infty$.
\end{ex}

\begin{ex}\label{LqLpMod}
Let $(\Omega, \mathfrak{M}, \mu)$ be a measure space, let $p \in (1, \infty)$, 
and consider the $L^p$-operator algebra $A=\mathcal{K}(L^p(\mu))$ of compact operators on $L^p(\mu)$.  As before, we let $q$ be the Hölder conjugate of $p$. We can switch the order on the modules in Example \ref{LpLqMOD} and still get 
an $L^p$-module but over $\mathcal{K}(L^p(\mu))$ instead of $\C$. Indeed, 
let $\X = L^q(\mu)$, identified as before with $\Li(L^p(\mu), \ell_1^p)$,
and let $\Y = L^p(\mu)$ which is identified again with $\Li(\ell_1^p, L^p(\mu))$. 
For any $a \in\Li(L^p(\mu))$ let $a' \in \Li(L^q(\mu))$ be the Banach dual map of $a$, which satisfies $\langle a'\eta, \xi \rangle = \langle \eta, a \xi \rangle$ for any $\xi \in \Y$ and $\eta \in \X$. It is straightforward to check that $\eta a = a'(\eta) \in \X$ for any $\eta \in \X$, whence Condition \eqref{lpm5} in Definition \ref{mainD} follows.  Condition \eqref{lpm5.1} follows at once from the fact that $A$ naturally acts on $L^p(\mu)$ on the left as bounded operators. Condition \eqref{lpm5.2} also holds, for a direct calculation shows that $\xi \eta= \theta_{\xi, \eta} \in \mathcal{K}(L^p(\mu)) =A$. Finally, since $\| \theta_{\xi, \eta}\| = \|\xi\|_p \|\eta\|_q$, it is also clear that $(L^p(\mu), L^q(\mu))$ is a C*-like 
module over $\mathcal{K}(L^p(\mu))$. 
\end{ex}

\begin{ex}\label{lqlpMOD}
Let $d \in \Z_{\geq 1}$, let $p \in [1, \infty)$
and let $q$ be the Hölder conjugate of $p$. 
As particular instance of Example \ref{LqLpMod} we get that 
 $(\ell_d^p, \ell_d^q)$ is a C*-like $L^p$-module 
 over $\mathcal{K}(\ell_d^p)=M_d^p(\C)$. 
 We are again able to include $p=1$ because 
the dual of $\ell_d^1$ is $\ell_d^\infty$.
\end{ex}

\begin{ex}\label{AdAdMOD}
In this example we combine, via the spatial tensor product, Example \ref{lplqMOD} 
with Example \ref{A_A}. This is a particular case of the external tensor product 
construction discussed in Section \ref{ETPC} below. 
Let $d \in \Z_{\geq 2}$, let $p \in (1, \infty)$, and let $(\Omega, \mathfrak{M},\mu)$
be a measure space. 
If $\nu_d$ is counting measure on $\{1, \ldots, d\}$, then we have 
the following isometric isomorphisms
\[
\ell_d^p \otimes_p L^p(\mu) \cong L^p(\nu_d \times \mu) \cong L^p(\mu)^d.
\]
The last isomorphism one comes from the map $\xi \mapsto (\xi_1, \ldots, \xi_d)$ where, for each $j \in \{1, \ldots, d\}$, 
 $\xi_j \in L^p(\mu)$ is given by $\xi_j(\omega)=\xi(j, \omega)$,
 and the norm on $ L^p(\mu)^d$ is given by 
 \[
 \| (\xi_1, \ldots, \xi_d)\|=\Bigl( \sum_{j=1}^d\|\xi_j\|^p\Bigl)^{1/p}.
 \]
 Now let $A \subseteq \Li(L^p(\mu))$
be an $L^p$-operator algebra. 
We define $\X \subseteq \Li(L^p(\mu), L^p(\mu)^d)$ and 
$\Y \subseteq \Li( L^p(\mu)^d, L^p(\mu))$
by 
 \[
 \X = \ell_d^p \otimes_p A =  \Li(\ell_1^p, \ell_d^p) \otimes_p A \ 
  \mbox{ and } \  \Y = \ell_d^q \otimes_p A =\Li(\ell_d^p, \ell_1^p) \otimes_p A. 
 \]
 Observe that $\X$ is identified with $A^d$, with norm given by
\[
\| (a_1, \ldots, a_d) \| = \sup_{\| \xi \| = 1 } \Bigl(\sum_{j=1}^d \|a_j\xi\|^p \Bigl)^{1/p},
\]
where the supremum is taken over $\xi \in L^p(\mu)$. Similarly,  $\Y$ is 
also identified with 
$A^d$, but equipped with the norm
\[
\| (b_1, \ldots, b_d) \| = \sup_{\| (\xi_1, \ldots, \xi_d)\| = 1 } \Bigl\| \sum_{j=1}^d b_j\xi_j\Bigl\|, 
\]
where the supremum is taken over 
$(\xi_1, \ldots, \xi_d) \in  L^p(\mu)^d$. 
 Since $\X$ and $\Y$ are closed by construction, we automatically 
 have the closure requirements of Definition \ref{mainD}. 
 For Condition \eqref{lpm5}, take $z \in \ell_d^p$ and $a_1, a_2 \in A$.  
 We have 
 \[
 (z \otimes a_1)a_2=z \otimes a_1a_2 \in \X.
 \]
Therefore the composition $xa$ is in $\X$ for all 
$x \in \X$ and all $a \in A$. 
Condition \eqref{lpm5.1} follows similarly. Indeed,
 if $a_1, a_2 \in A$ and $w \in \ell^q_d$ we get 
\[
a_1 (w \otimes a_2) = w \otimes a_1a_2 \in \Y,
\]
whence $ay \in \Y$ for all $a \in A$ and $y \in \Y$. 
To verify Condition \eqref{lpm5.2}, 
notice that for $w \in \ell_d^q$, 
$z \in \ell_d^p$, and $a_1, a_2 \in A$, we have 
\[
 (w \otimes a_1)(z \otimes a_2)=\Bigl( \sum_{j=1}^d w(j)z(j)\Bigl)\ a_1a_2 \in A.
\]
Hence, it follows that $\lpp{y}{x}{A} \in A$ for all $y \in \Y$ and all $x \in \X$.
Thus, $(\Y, \X)$ is
an $L^p$-module over $A$.
The C* likeness of $(\Y, \X)$ for certain $A \subseteq M_k^p(\C)$ is studied in detail in \cite{REU2024}, where it is shown that $(\Y, \X)$ is C*-like when $A$ is any block diagonal subalgebra of $M_k^p(\C)$. However,  $(\Y, \X)$ is generally not C*-like, not even if $A$ is unital.
\end{ex}

\subsection{Finite Direct Sum of $L^p$-modules}

Let $p \in (1, \infty)$. Example \ref{AdAdMOD} can be realized as the direct sum of 
$d$ copies of the 
$L^p$-module from Example \ref{A_A}. 
We now describe such direct sum in full generality.  
Let $p\in [1, \infty)$, let $d \in \Z_{\geq 2}$, and for each $j \in \{1, \ldots, d\}$
let $(\X_j, \Y_j)$ be an $L^p$-module
over an $L^p$-operator algebra 
$A \subseteq \Li(L^p(\mu_0) )$. For $j \in \{1, \ldots, d\}$, we have 
measure spaces $(\Omega_j, \mathfrak{M}_j, \mu_j)$ such that 
$\X_j$ is a closed subspace of 
$\Li(L^p(\mu_0), L^p(\mu_j) )$ and 
$\Y_j$ is a closed subspace of 
$\Li(L^p(\mu_j), L^p(\mu_0) )$. 
Consider the algebraic direct sums $\X=\bigoplus_{j=1}^d \X_j$ 
and $\Y=\bigoplus_{j=1}^d\Y_d$.
The pair $( \X, \Y)$ has a natural structure of 
$L^p$-module over $A$. Indeed, 
\[
\X \subseteq \Li\Bigl( L^p(\mu_0), 
\bigoplus_{j=1}^d  L^p(\mu_j) \Bigl),
\]
 where each $(x_1, \ldots, x_d) \in \X$ acts on 
 $\xi \in L^p(\mu_0)$ by 
 \[
(x_1, \ldots, x_d)  \xi = (x_1\xi, \ldots, x_d\xi).
 \]
This endows $\X$ with the operator norm satisfying
\[
\max_{j=1, \ldots, d}\|x_j\| \leq \| (x_1, \ldots, x_d)  \| 
\leq \Bigl(\sum_{j=1}^d \|x_j\|^p \Bigl)^{1/p}.
\]
Even though in general neither equality is true, this shows that 
$\X$ is a closed subspace of  $\Li\big( L^p(\mu_0), 
\bigoplus_{j=1}^d  L^p(\mu_j)  \big).$ Similarly, 
\[
\Y \subseteq \Li\Bigl(\bigoplus_{j=1}^d  L^p(\Omega_j, \mu_j) ,
  L^p(\Omega_0,\mu_0)  \Bigl) 
\]
 where each $(y_1, \ldots, y_d) \in \Y$ acts on 
 $(\eta_1, \ldots,\eta_d) \in \bigoplus_{j=1}^d  L^p(\mu_j)$ by 
 \[
(y_1, \ldots, y_d)(\eta_1, \ldots,\eta_d)  = \sum_{j=1}^d y_j\eta_j.
 \]
Thus, the operator norm inherited by $\Y$ satisfies
\[
\max_{j=1, \ldots, d} \|y_j\| \leq \| (y_1, \ldots, y_d)  \| 
\leq \Bigl(\sum_{j=1}^d \|y_j\|^q \Bigl)^{1/q}.
\]
where $q$ is the Hölder conjugate for $p$.
 Once again, equality in both ends of the last inequality
does not always hold, 
but it follows that $\Y$ is a closed subspace of  $\Li\big(  
 \bigoplus_{j=1}^d L^p(\mu_j), L^p(\mu_0) \big)$. 
For each $(x_1, \ldots, x_d) \in \X$ and $a \in A$,
it is clear that Condition \eqref{lpm5} in Definition \ref{mainD} holds:
\[
(x_1, \ldots, x_d) a=(x_1a, \ldots, x_da) \in \X
\] 
We now check condition \eqref{lpm5.1} . 
Indeed, it is clear that if $(y_1, \ldots, y_d) \in \Y$, $a \in A$,
then $ay_j \in \Y_j$ for each $j \in\{1, \ldots, d\}$, and therefore
we have 
\[
a(y_1, \ldots, y_d)=(ay_1, \ldots, ay_d) \in  \Y.
\]
For Condition \eqref{lpm5.2}, if $(y_1, \ldots, y_d) \in \Y$,
we get 
\[
(y_1, \ldots, y_d)(x_1, \ldots, x_d) =\sum_{j=1}^d \lpp{y_j}{x_j}{A} \in A,
\]
whence $(\Y, \X)$ is an $L^p$-module over $A$. 

\subsection{Countable Direct Sums of $L^p$-modules}

We start by discussing a naive attempt at defining
countable direct sums of $L^p$-modules that generalizes 
the finite dimensional case. We then give an example to show
why this fails in general. We finish the section with the correct
definition and a result that shows that this definition
 generalizes direct sums of Hilbert modules. 

Let $p \in [1, \infty)$. Suppose now that we have a sequence of measure spaces 
$((\Omega_j, \mathfrak{M}_j, \mu_j))_{j=0}^\infty$
and a sequence of $L^p$-modules $((\Y_j, \X_j))_{j=1}^\infty$ 
over an $L^p$-operator algebra $A \subseteq \Li(L^p(\mu_0))$ such that, 
for each $j \in \Z_{\geq 1}$,
the module $\X_j$ is a closed subspace of $\Li(L^p(\mu_0), L^p(\mu_j))$. 
An immediate generalization from the finite case is to consider
\begin{align*}
\X_{\op{w}}& = \left\{ (x_{j})_{j=1}^\infty \colon x_j \in \X_j, \sup_{\|\xi\|_p=1} \sum_{j=1}^\infty \|x_j\xi\|^p < \infty\right\},\\
\Y_\op{w} & =  \left\{ (y_{j})_{j=1}^\infty \colon y_j \in \Y_j, \sup_{\sum_{j=1}^\infty\|\eta_j\|^p_p=1} \Big\|\sum_{j=1}^\infty y_j\eta_j\Big\|_p < \infty\right\},
\end{align*}
where the supremum for elements in $\X_{\op{w}}$ is taken over elements $\xi \in L^p(\mu_0)$
and the one for elements in $\Y_\op{w}$ is taken considering elements $\eta_j \in L^p(\mu_j)$
for each $j \in \Z_{\geq 1}$. If we equip $ \bigoplus_{j=1}^\infty L^p(\mu_j)$ with the usual $p$-norm, then $\X_{\op{w}}$ is a closed subspace of $\Li\big(L^p(\mu_0), \bigoplus_{j=1}^\infty L^p(\mu_j)\big)$ and 
$\Y_{\op{w}}$ is a closed subspace of  $\Li\big(\bigoplus_{j=1}^\infty L^p(\mu_j), L^p(\mu_0) \big)$ (this 
will follow from Theorem \ref{CDsumisLpmod}). 
Furthermore, the pair $( \Y_{\op{w}}, \X_{\op{w}} )$ satisfies Conditions \eqref{lpm5} and \eqref{lpm5.1} in 
Definition \ref{mainD}. However, Condition \eqref{lpm5.2} fails in general. Indeed, in the following example we will see that
it is not always true that 
requiring $(x_j)_{j=1}^\infty \in \X_{\op{w}}$ and $(y_{j})_{n=1}^\infty \in \Y_{\op{w}}$ 
implies that
\[
(y_{j})_{n=1}^\infty (x_j)_{j=1}^\infty = \sum_{j=1}^\infty \lpp{y_j}{x_j}{A} 
\]
converges to an element of $A$. 

\begin{ex}
Let $p \in [1, \infty)$ and consider $(\ell^p,\ell^q)$,
which is a C*-like $L^p$-module over $\mathcal{K}(\ell^p)$, as shown in Example \ref{LqLpMod}
(we are able to include $p=1$ because the dual of $\ell^1$ is $\ell^\infty$). 
For each $j \in \Z_{\geq 1}$ we let $(\Y_j, \X_j) = (\ell^p, \ell^q)$ 
and consider $\X_{\op{w}}$ and $\Y_{\op{w}}$ as above. For each $j \in \Z_{\geq 1}$ define 
$x_j \colon \ell^p \to \ell_1^p$ by $x_j\xi = \xi(j)$ 
and $y_j \colon \ell_1^p \to \ell^p$ by $y_j\zeta = \zeta \delta_j$, 
where $\{\delta_j \colon j \in \Z_{\geq 1}\}$ is the canonical basis of $\ell^p$ (notice that for $p=2$, 
$y_j$ is actually $x_j^*$).  
Then $x_j \in \X_j$ and $y_j \in \Y_j$ for each $j \geq 1$. Furthermore, 
\[
\sup_{\|\xi\|_p=1} \sum_{j=1}^\infty |x_j\xi|^p = \sup_{\|\xi\|_p=1} \|\xi\|_p^p = 1,
\]
and 
\[
\sup_{\sum_{j=1}^\infty|\zeta_j|^p_p=1} \Big\|\sum_{j=1}^\infty y_j\zeta_j\Big\|^p_p = \sup_{\sum_{j=1}^\infty|\zeta_j|^p=1} \sum_{j=1}^\infty|\zeta_j|^p =1.
\]
Therefore $(x_j)_{j=1}^\infty \in \X_{\op{w}}$ and $(y_j)_{j=1}^\infty \in \Y_{\op{w}}$. 
Moreover, for each $j \in \Z_{\geq 1}$ we clearly have 
$ y_jx_j\xi = \xi(j)\delta_j$ and therefore $y_jx_j  = \theta_{\delta_j, \delta_j} \in \mathcal{K}(\ell^p)$. However, 
\[
\Bigl \| \sum_{j=n}^m \theta_{\delta_j,\delta_j} \Bigl\| =1
\]
for any $m \geq n \geq 1$, and therefore $\sum_{j=1}^\infty y_jx_j = \sum_{j=1}^\infty \theta_{\delta_j,\delta_j} $
does not converge in $\mathcal{K}(\ell^p)$. 
\end{ex}

Thus, in general $(\Y_\op{w}, \X_{\op{w}})$ is 
not an $L^p$-module over $A$. We actually need to work with subspaces of 
$\X_{\op{w}}$ and $\Y_{\op{w}}$ to make things work. 
The motivation for the following definition for 
countable direct sums of $L^p$-modules will be clear 
once we introduce the external tensor product in Section \ref{ETPC} and 
prove Proposition \ref{stdL^pMod}.

\begin{defi}\label{countablesum}
Let $p \in [1, \infty)$,  for each $j \in \Z_{\geq 0}$ let  
$(\Omega_j, \mathfrak{M}_j, \mu_j)$ be a measure space, and 
let $((\Y_j, \X_j))_{j=1}^\infty$ be a sequence of $L^p$-modules over an $L^p$-operator algebra 
$A \subseteq \Li(L^p(\mu_0))$ such that for $j \in \Z_{\geq 1}$, 
the module $\X_j$ is a closed subspace of $\Li(L^p(\mu_0), L^p(\mu_j))$. 
We define the direct sum module $\bigoplus_{j=1}^\infty (\Y_j, \X_j)$ as the pair 
$(\Y, \X)$ where  
\begin{align*}
\X & 
= \left\{
 (x_{j})_{j=1}^\infty \in \X_{\op{w}}\colon \lim_{n,m \to \infty} \sup_{\| \xi \|_p=1 } \sum_{j=n}^m \|x_j\xi \|_p^p = 0 
 \right\},\\
\Y & =  \left\{
 (y_{j})_{j=1}^\infty \in \Y_\op{w} \colon \lim_{n,m \to \infty}\sup_{\sum_{j=1}^\infty \|\eta_j\|^p_p=1} \Big\|\sum_{j=n}^m y_j\eta_j\Big\|_p =0
 \right\}.
\end{align*}
\end{defi}

Next, we show that $\bigoplus_{j =1}^\infty (\Y_j, \X_j)$ is indeed 
an $L^p$-module over $A$ that agrees with the usual definition of direct sum 
of Hilbert modules when $A$ is a C*-algebra.

\begin{theorem}\label{CDsumisLpmod}
Let $(\Y, \X) = \bigoplus_{j = 1}^\infty (\Y_j,\X_j)$ be as in 
Definition \ref{countablesum}. Then:
\begin{enumerate}
\item  $(\Y, \X)$ is an $L^p$-module over $A$. \label{DSisLpMOD}
\item  Let $p=2$,  let $A$ be a C*-algebra, and for each $j \geq 1$ 
let $\X_j$ be a Hilbert $A$-module isometrically represented in $(\Hi_0, \Hi_j)$ via $\pi_{\X_j} \colon \X_j \to \Li(\Hi_0, \Hi_j)$, as in Definition 3.7 in \cite{Delfin_2022}, with $\X_j \Hi_0$ dense in $\Hi_j$. 
Then 
\[
\big(  (\pi_{\X_j}(x_j)^*)_{j=1}, (\pi_{\X_j}(x_j))_{j=1}  \big) \in \bigoplus_{j = 1}^\infty ( \pi_{\X_j}(\X_j)^*, \pi_{\X_j}(\X_j) ) 
\]
if and only if $(x_j)_{j=0}^\infty \in \bigoplus_{j = 1}^\infty  \X_j$ (that is if and only if $\sum_{j =1}^\infty \langle x_j, x_j\rangle_A$ converges in $A$). 
\end{enumerate}
\end{theorem}
\begin{proof}
To prove the first statement, we first check that $\X$ is a closed 
subspace of $\Li\big( L^p(\mu_0), \bigoplus_{j=1}^\infty L^p(\mu_j) \big)$
and that $\Y$ is a closed subspace of  $\Li\big(  \bigoplus_{j=1}^\infty L^p(\mu_j), L^p(\mu_0) \big)$. 
To do so, let $(x^{(n)})_{n=1}^\infty$ be a Cauchy sequence in $\X$. Then a direct 
check shows that for each $j \in \Z_{\geq 1}$, $\| x_j^{(n)}-x_j^{(m)}\| \leq \| x^{(n)}-x^{(m)}\|$ and 
therefore $(x_j^{(n)})_{n=1}^\infty$ is a Cauchy sequence in $\X_j$. 
Thus, by completeness, we get for each $j \in \Z_{\geq 1}$ an element $x_j \in \X_j$
such that $x^{(n)}_j \to x_j$  as $n \to \infty$. Define $x=(x_j)_{j=1}^\infty$. We claim that 
$(x^{(n)})_{n=1}^\infty$ converges to $x$. Let $\varepsilon>0$ 
and choose $N\in \Z_{\geq 1}$ such that 
$\| x^{(n)}- x^{(m)} \|^p < \varepsilon^p$ whenever $m\geq n \geq N$. 
Now take any $\xi \in L^p(\mu_0)$ with $\| \xi\|=1$, and observe that
 \[
 \sum_{j=1}^\infty \| (x_j^{(n)}-x_j^{(m)})\xi \|^p \leq \| x^{(n)}- x^{(m)} \|^p < \varepsilon^p.
 \]
 Letting $m \to \infty$ on both ends of the previous inequality gives 
 \[
  \sum_{j=1}^\infty \| (x_j^{(n)}-x_j)\xi \|^p < \varepsilon^p.
 \]
Taking supremum over all $\|\xi\|=1$ yields $\| x^{(n)}-x\|^p <\varepsilon^p$ whenever $n \geq N$. 
Thus, $x^{(n)}$ converges to $x$. 
Similarly, if we let $(y^{(n)})_{n=1}^\infty$ be a Cauchy sequence in $\Y$, 
for each $j$ we see that $(y_j^{(n)})_{n=1}^\infty$ is a Cauchy sequence in $\Y_j$ 
and therefore we get an element $y_j \in \Y_j$ such that $y^{(n)}_j \to y_j$. A similar 
argument shows that, if we define $y=(y_j)_{j=1}^\infty$, 
 then $y^{(n)}$ converges to $y$. 
It remains to check that $x \in \X$ and $y \in \Y$. 
 For any $\xi \in L^p(\mu_0)$ with $\| \xi\|=1$ and for any $m > n \geq 1$ we repeatedly apply Minkowski's inequality (both for $L^p(\mu_j)$ and for $\R^{m-n}$) to get 
 \begin{align*}
 \Big( \sum_{j=n}^m \| x_j\xi \|^p\Big)^{1/p} 
 & \leq  \Big( \sum_{j=n}^m (\| x_j\xi -x_j^{(k)}\xi\|^p\Big)^{1/p} +\Big( \sum_{j=n}^m \|x_j^{(k)}\xi\|^p\Big)^{1/p} \\
  & \leq  \Big( \sum_{j=n}^m (\| x_j\xi -x_j^{(k)}\xi\|^p\Big)^{1/p} +\Big( \sum_{j=n}^m \|x_j^{(k)}\xi\|^p\Big)^{1/p} \\
  & \leq \| x-x^{(k)} \| + \Big( \sum_{j=n}^m \|x_j^{(k)}\xi\|^p\Big)^{1/p}.
 \end{align*}
Since $x^{(k)} \in \X$, it now follows that $x \in \X$, proving closure of $\X$. 
Similarly, if $(\eta_j)_{j=1}^\infty$ is a norm one element of $\bigoplus_{j=1}^\infty L^p(\mu_j)$ and 
$m \geq n \geq 1$, a direct application of Minkowski's inequality in $L^p(\mu_0)$ gives 
\[
\Big\| \sum_{j=n}^m y_j\eta_j \Big\| \leq \Big\| \sum_{j=n}^m (y_j-y^{(k)}_j)\eta_j\Big\|+ \Big\| \sum_{j=n}^m y^{(k)}_j\eta_j\Big\|  
\leq \|y-y^{(k)}\| + \Big\| \sum_{j=n}^m y^{(k)}_j\eta_j\Big\|.
\]
Hence, since $y^{(k)} \in \Y$, it follows that $y \in \Y$, proving that $\Y$ is also closed.

It still remains for us to check that conditions \eqref{lpm5}-\eqref{lpm5.2} in 
Definition \ref{mainD} are satisfied. Condition \eqref{lpm5.2} is the only one that requires some work. 
Let $(x_j)_{j=1}^\infty \in \X$ and 
$(y_j)_{j=1}^\infty \in \Y$, we only need to check that $\sum_{j=1}^\infty \lpp{y_j}{x_j}{A}$ converges in $A$ to the operator $\lpp{(y_j)_{j=1}^\infty }{(x_j)_{j=1}^\infty }{A} \colon L^p(\mu_0) \to  L^p(\mu_0)$.
Indeed,  set 
\[
K=\sup_{\sum_{j=1}^\infty \| \eta_j \|^p = 1}\Big\| \sum_{j=1}^\infty y_j \eta_j \Big\|,
\]
 and for each $m\geq n \geq 1$ let $M_{n,m}(\xi) = \sum_{j=n}^m \| x_j\xi\|^p$. Then 
$K < \infty$ and $\lim_{m,n \to\infty} \sup_{\|\xi\|=1}M_{n,m}(\xi) =0$.  Now
for any $\xi \in L^p(\mu_0)$ with $\|\xi\|=1$, we have
\[
\Big\| \sum_{j=n}^m y_jx_j \xi \Big\| \leq K M_{m,n}(\xi).
\]
Hence, 
\[
\Big\| \sum_{j=n}^m y_jx_j\Big\| \leq K\sup_{\|\xi\|=1} M_{m,n}(\xi),
\]
from which it follows that $(\sum_{j=1}^n y_jx_j)_{n=1}^\infty$ is a Cauchy sequence
in $A$ and therefore converges to $\lpp{(y_j)_{j=1}^\infty }{(x_j)_{j=1}^\infty }{A}$, proving Condition \eqref{lpm5.2}. This proves Part \eqref{DSisLpMOD} in the statement.

For the second part of the statement, the `only if' implication follows immediately from the fact that $(\Y, \X)$ is an $L^2$-module over $A$ thanks to Part \eqref{DSisLpMOD}. For the `if' implication, identify $A$ with its isometric copy in 
$\Li(\Hi_0)$ and similarly for each $j \in \Z_{\geq 1}$ we identify $\X_j$
with its isometric copy in $\Li(\Hi_0, \Hi_j)$ so that $\X_j^* \subseteq
 \Li(\Hi_j, \Hi_0)$. We have to show that convergence of $\sum_{j=1}^\infty x_j^* x_j$
 in $A$ implies the following two conditions 
\begin{enumerate}
\item[(a)] $\displaystyle{ \sup_{\| \xi \|_2=1 } \sum_{j=1}^\infty \|x_j\xi \|_2^2 < \infty}$ and $\displaystyle{\lim_{n,m \to \infty}   \sup_{\| \xi \|_2=1 } \sum_{j=n}^m \|x_j\xi \|_2^2 = 0}$,
\item[(b)] $\displaystyle{  \sup_{\sum_{j=1}^\infty \|\eta_j\|^2_2=1} \Big\|\sum_{j=1}^\infty x^*_j\eta_j\Big\|_2 < \infty}$ and $\displaystyle{ \lim_{n,m \to \infty} \sup_{\sum_{j=1}^\infty \|\eta_j\|^2_2=1} \Big\|\sum_{j=n}^m x^*_j\eta_j\Big\|_2 =0.}$
\end{enumerate}
 To check condition (a), let $\xi \in \Hi_0$
 have norm $1$ and let 
$m\geq n \geq 1$. Then, 
 \[
  \sum_{j=n}^m \|x_j\xi \|_2^2  =   \sum_{j=n}^m \langle \xi, x_j^*x_j\xi \rangle 
  =  \Big\langle \xi, \sum_{j=n}^m x_j^*x_j\xi \Big\rangle \leq \Big\| \sum_{j=n}^m x_j^*x_j \Big\|,
 \]
 and also 
 \[
   \sum_{j=1}^\infty \|x_j\xi \|_2^2  \leq  \Big\| \sum_{j=1}^\infty x_j^*x_j \Big\|.
 \]
 Hence, convergence of $\sum_{j=1}^\infty x_j^* x_j$ in $A$ does imply condition (a). 
 For condition (b), let $(\eta_j)_{j=1}^\infty$ be a norm
 $1$ element of $\bigoplus_{j=1}^\infty \Hi_j$. 
 In addition, for fixed $m\geq n \geq 1$,
define $\bm{\eta}=(\eta_n, \ldots, \eta_m) \in \bigoplus_{j=n}^m \Hi_j$. 
Observe 
 that $\|\bm{\eta}\| \leq \| (\eta_j)_{j=1}^\infty \|_2 =1$. Then
 \begin{align*}
 \Big\| \sum_{j=n}^m x_j^*\eta_j \Big\|^2 
 & = 
 \Big\langle  \sum_{j=n}^m x_j^*\eta_j , \sum_{k=n}^m x_k^*\eta_k \Big\rangle \\
& = \sum_{j = n}^m \sum_{k=n}^m \langle \eta_j,  x_jx_k^*\eta_k \rangle \\
& = \langle \bm{\eta}, (x_jx_k^*)_{j,k=n}^m\bm{\eta} \rangle \\
& \leq \|  (x_jx_k^*)_{j,k=n}^m\|.
 \end{align*}
Both statements in condition (b) now follow at once from the convergence of 
$\sum_{j=1}^\infty x_j^*x_j $ and 
the fact that the norms   $\| (x_jx_k^*)_{j,k=n}^m  \|$ and $\| \sum_{j=n}^m x_j^*x_j \|$
agree (see either Lemma 2.1 in \cite{KajPinWat1998} or Lemma 3.1.6 in \cite{Del2023} for a proofs of this norm equality).
\end{proof}

\subsection{External Tensor Product of $L^p$-modules}\label{ETPC}

We now present an analogue of the external tensor product 
for Hilbert modules. This generalizes the construction 
from Example \ref{AdAdMOD}. Moreover, Proposition 
\ref{stdL^pMod} below was in fact the main motivation for the 
definition 
of countable direct sums presented above (see Definition \ref{countablesum}). 

\begin{defi}\label{ExtTP}
For $j \in \{0,1\}$, let $(\Omega_j, \mathfrak{M}_j, \mu_j)$ and
$(\Lambda_j, \mathfrak{N}_j, \nu_j)$ be measures spaces, 
let $p \in (1, \infty)$, let $(\Y,\X)$ be an $L^p$-module over an $L^p$-operator algebra 
$A \subseteq \Li(L^p(\mu_0))$ with $\X \subseteq \Li(L^p(\mu_0), L^p(\mu_1))$, 
and let $(\W,\V)$ be an $L^p$-module over an $L^p$-operator algebra 
$B \subseteq \Li(L^p(\nu_0))$ with $\V \subseteq \Li(L^p(\nu_0), L^p(\nu_1))$. 
Using the spatial tensor product for operators 
acting on $L^p$-spaces, we define the \emph{external tensor product of $(\Y, \X)$ with $(\W, \V)$} by letting 
\[
(\Y, \X) \otimes_p (\W, \V) = (\Y \otimes_p \W, \X \otimes_p \V).
\]
\end{defi}

A routine check shows that all the conditions in Definition \ref{mainD} needed to make 
$(\Y \otimes_p \W, \X \otimes_p \V)$ an $L^p$-module over $A \otimes_p B$
are met. 

\begin{tpr}\label{stdL^pMod}
Let $p \in [1, \infty)$, let $q$ be its Hölder conjugate, and let $(\Y, \X)$ be an $L^p$-module over $A\subseteq \Li(L^p(\mu_0))$
with $\X \subseteq \Li(L^p(\mu_0), L^p(\mu_1))$. 
Then, if $(\ell^q, \ell^p)$ is the C*-like module over $\C$ from Example \ref{LpLqMOD}
(here $\Omega=\Z_{\geq 1}$), we have 
\[
(\ell^q, \ell^p) \otimes_p (\Y, \X) = \bigoplus_{j = 1}^\infty( \Y, \X).
\]
\end{tpr}
\begin{proof}
Recall that $\bigoplus_{j = 1}^\infty( \Y, \X) = (\Zm_\Y, \Zm_\X)$
where
\[
\Zm_\X = \left\{ (x_j)_{j=1}^\infty \colon x_j \in \X, \lim_{n,m \to \infty} \sup_{\| \xi \|_p=1 } \sum_{j=n}^m \|x_j\xi \|_p^p = 0\right\} \subseteq \Li\Big(L^p(\mu_0), \bigoplus_{j=1}^\infty L^p(\mu_1)\Big),
\]
and
\[
\Zm_\Y = \left\{ (y_j)_{j=1}^\infty \colon y_j \in \Y, \lim_{n,m \to \infty} \sup_{
\sum_{j=1}^\infty\|\eta_j\|^p_p=1}\Big\| \sum_{j=n}^m y_j\eta_j \Big\| = 0\right\}
\subseteq \Li\Big(\bigoplus_{j=1}^\infty L^p(\mu_1), L^p(\mu_0)\Big).
\]
Denote by $\iota_\X$ and $\iota_\Y$ the following isometric inclusions:
\[
 \iota_\X \colon \ell^p \otimes_p \X \to \Li\Big(L^p(\mu_0), \bigoplus_{j=1}^\infty L^p(\mu_1)\Big),
 \]
  and 
  \[
  \iota_\Y \colon \ell^q \otimes_p \Y \to \Li\Big(\bigoplus_{j=1}^\infty L^p(\mu_1), L^p(\mu_0)\Big).
  \]
It suffices to show that the image of $\iota_\X$ is $\Zm_\X$
and that the image of $\iota_\Y$ is $\Zm_\Y$. For any $\zeta \in \ell^p$, any 
$x \in \X$, and any $\xi \in L^p(\mu_0)$ we have $\iota_\X(\zeta \otimes x)\xi = (\zeta(j)x\xi)_{j=1}^\infty \in 
\bigoplus_{j=1}^\infty L^p(\mu_1)$. Furthermore,
\[
\lim_{m,n \to \infty} \sup_{\| \xi\|=1} \sum_{j=n}^m \| \zeta(j)x\xi \|^p = \|x\|\lim_{m,n \to \infty} \sum_{j=n}^m | \zeta(j)|^p=0.
\]
From this it is clear that $\iota_\X(\xi \otimes x) \in \Zm_{\X}$. Since $\Zm_\X$ is closed in $\Li\big(L^p(\mu_0), \bigoplus_{j=1}^\infty L^p(\mu_1)\big)$ (see Theorem \ref{CDsumisLpmod}), 
we conclude that $\iota_\X( \ell^p \otimes_p \X) \subseteq \Zm_\X$. For the reverse inclusion, suppose that $(x_j)_{j=1}^\infty$ is in $\Zm_\X$. We claim that $\sum_{j=1}^\infty \delta_j \otimes x_j$ is an element 
of $\ell^p \otimes \X$. Indeed, for any $m \geq n \geq 1 $ we have 
\[
\Big\| \sum_{j=n}^m \delta_j \otimes x_j \Big\|^p = \sup_{\| \xi\|=1} \sum_{k=1}^\infty \int_{\Omega_1} \Big|  \sum_{j=n}^m \delta_j(k) (x_j\xi)(\omega)\Big|^p d\mu_1(\omega)  =\sup_{\| \xi\|=1}  \sum_{j=n}^m \|x_j\xi\|^p.
\]
After taking the limit as $m,n \to \infty$
and using the fact that $(x_j)_{j=1}^\infty$ is in $\Zm_\X$,
we see that $(\sum_{j=1}^n \delta_j \otimes x_j )_{n=1}^\infty$ is a Cauchy sequence in 
$\ell^p \otimes \X$, so our claim follows. It is immediate 
to check that $\iota_\X(\sum_{j=1}^\infty \delta_j \otimes x_j) = (x_j)_{j=1}^\infty$, and therefore 
we have shown that $\iota_\X( \ell^p\otimes_p \X) =\Zm_\X$ as wanted. 
Similarly, notice that for any $\upsilon \in \ell^q$, $y \in \Y$, and 
$(\eta_{j})_{j=1}^\infty \in \bigoplus_{j=1}^\infty L^p(\mu_1)$
we have
\[
\iota_\Y(\upsilon \otimes y)(\eta_{j})_{j=1}^\infty  = \sum_{j=1}^\infty \upsilon(j)y\eta_j.
\]
Hence, using the finite dimensional version of Hölder's inequality we see that 
for $m \geq n \geq 1$, 
\[
\sup_{
\sum_{j=1}^\infty\|\eta_j\|^p_p=1} \Big\| \sum_{j=n}^m \upsilon(j)y\eta_j \Big\|
\leq \| y \|  \Big( \sum_{j=n}^m |\upsilon(j)|^q\Big)^{1/q}.
\]
Thus, taking limit when $m,n \to \infty$ shows that 
$\iota_Y(\upsilon \otimes y) \in \Zm_\Y$.  Since $\Zm_\Y$ is closed, 
this is enough to show that $\iota_Y(\ell^q\otimes_p \Y) \subseteq \Zm_\Y$. 
For the reverse inclusion, once again it suffices to show that $\sum_{j=1}^\infty \delta_j \otimes y_j$ 
defines an element in $\ell^q \otimes_p \Y$ when $(y_j)_{j=1}^\infty \in \Zm_\Y$.
Let $m \geq n \geq 1 $ and notice that
\[
\Big\|\sum_{j=n}^m \delta_j \otimes y_j \Big\| =  \sup_{\sum_{k=1}^\infty\|\eta_k\|^p_p=1} 
\Big\|\sum_{k=1}^\infty\sum_{j=n}^m \delta_j(k) y_j\eta_k\Big\|
=\sup_{\sum_{k=1}^\infty\|\eta_k\|^p_p=1} 
\Big\|\sum_{k=n}^m y_k\eta_k\Big\|.
\]
Thus, letting $m,n \to \infty$ shows that $(\sum_{j=1}^n \delta_j \otimes y_j )_{n=1}^\infty$
is Cauchy in $\ell^q \otimes_p \Y$ and we are done. 
\end{proof}

Recall from classic Hilbert module theory that 
the \emph{standard Hilbert $A$-module} of 
a C*-algebra $A$ is given by $\bigoplus_{\Z_{\geq 1}} A$, 
which is the same as $\ell^2 \otimes_\C A$.
Below we use Proposition \ref{stdL^pMod} to define the the $L^p$-analogue of this module. 

\begin{defi}\label{LPSTDMOD}
Let $p \in [1,\infty)$ and let $A$ be an $L^p$-operator algebra. 
We define \textit{the standard $L^p$-module 
of $A$} as
\[
(\ell^q, \ell^p)\otimes_p (A, A) = (\ell^q \otimes_p A, \ell^p \otimes_p A) = \bigoplus_{j=1}^\infty (A,A)
\]
where $(A,A)$ is the $L^p$-module from Example \ref{A_A}.
\end{defi}

\subsection{Morphisms of $L^p$-modules}

In this section we define morphism between $L^p$-modules over a fixed $L^p$-operator algebra $A$. 
The definitions and results here are motivated by the C*-case. In particular, the main definitions are modeled after Proposition 3.10 in \cite{Delfin_2022}, where a representation 
for adjointable maps of a Hilbert module is given. 

For the rest of the section we fix measure spaces $(\Omega_j, \mathfrak{M}_j, \mu_j)$, $j=0,1,2$ and 
we often let $E_j=L^p(\mu_j)$. We also fix an $L^p$-operator algebra $A \subseteq \Li(E_0)$ for $p \in [1, \infty)$, and two $L^p$-modules, $(\Y, \X)$ and $(\W, \V)$, over $A$ with $\X \subseteq \Li(E_0, E_1)$
and $\V \subseteq \Li(E_0, E_2)$.
\begin{defi}\label{GenMorph}
We define 
the space of $L^p$-module morphisms from $(\Y, \X)$ to  $(\W, \V)$ by
\[
\Li_A((\Y, \X) \to (\W, \V)) =
\{ t \in \Li(E_1,E_2) \colon tx \in \V, wt \in \Y, \text{ for all } x \in \X, w \in \W\}.
\]
\end{defi}
The main advantage of this definition is that  $\Li_A((\Y, \X) \to (\W, \V)) $ is, by construction, 
a subspace of operators on $L^p$-spaces. Furthermore, notice that any $t \in \Li_A((\Y, \X) \to (\W, \V))$ gives rise to a pair of maps $(t^l, t^r)$, 
$t^l \colon \W \to \Y$ and $t^r \colon \X \to \V$, defined by 
\[
t^l(w) = wt, \ t^r(x) = tx.
\] 
The next lemma 
shows that the pair $(t^l, t^r)$ is in fact a `linear operator' 
from the Banach $A$-pair $(\Y,\X)$ to the Banach $A$-pair $(\W, \V)$ as defined in Section 1
of \cite{paravicini_2009} (see Remark \ref{BanPA1}).
\begin{lem}\label{BanachMORPH}
Let $t \in \Li_A((\Y, \X) \to (\W, \V))$, $x \in \X$, $y \in \Y$, $v \in \V$, $w \in \W$, and $a \in A$. 
Then 
\begin{enumerate}
\item $t^l(aw)=at^l(w)$, $t^r(xa)=t^r(x)a$, and $\lpp{t^l(w)}{x}{A}=\lpp{w}{t^r(x)}{A}$;
\item If $(\Zm,\Um)$ is another $L^p$-module over $A$ and $s \in \Li_A((\W, \V) \to (\Zm, \Um))$, then $st \in \Li_A((\Y, \X) \to (\Zm, \Um))$ and $(st)^l=t^ls^l$, $(st)^r=s^rt^t$;
\item the composition $vy \in \Li_A((\Y, \X) \to (\W, \V))$.
\end{enumerate} 
\end{lem}
\begin{proof}
The first two parts are immediate from the definitions. For the third one, it is clear that $vy \in \Li(E_1,E_2)$. Further $(vy)x = v \lpp{y}{x}{A} \in \V$ for all
$x \in \X$ and also $w(vy) = \lpp{w}{v}{A}y \in \Y$ for all $w \in \W$. 
\end{proof}
We will often denote the operator $vy$ by $\theta_{v,y} \in \Li_A((\Y, \X) \to (\W,\V))$,
which in fact satisfies 
 \begin{align*}
  \theta^{l}_{v,y}(w) &=  \lpp{w}{v}{A}y \in \Y  \  \text{for all }w \in \W, \\
   \theta^{r}_{v,y}(x) &= v \lpp{y}{x}{A} \in \V  \  \text{for all }z \in \X.
 \end{align*}
 \begin{lem}\label{K_ID}
Let $t \in \Li_A((\W,\V) \to (\Zm, \Um) )$ and $s \in \Li_A((\Zm, \Um) \to (\Y, \X))$, where $(\Zm,\Um)$ is any $L^p$-module over $A$. Then,  for any $v \in V$, $y \in \Y$, 
 \[
 t\theta_{v,y} = \theta_{t^r(v), y}, \theta_{v,y}s = \theta_{v,s^l(y)}
 \]
 \end{lem}
 \begin{proof}
 Both equalities follow from a routine calculation. 
 \end{proof}
  We now can define the compact $L^p$-module maps:
  \begin{defi}\label{GenKMorph}
We define 
the space of compact $L^p$-module morphisms from $(\Y, \X)$ to  $(\W, \V)$ by
\[
\mathcal{K}_A((\Y, \X) \to (\W, \V)) = \cj{\op{span}\{\theta_{v,y} \colon v \in \V \mbox{ and } y \in \Y \} } \subseteq \Li_A((\Y, \X) \to (\W, \V)).
\]
\end{defi}
Once again, 
$\mathcal{K}_A((\Y, \X) \to (\W, \V))$ is naturally a space of operators on $L^p$-spaces. When $(\Y, \X)=(\W,\V)$ we put $\Li_A(\Y, \X)=\Li_A((\Y, \X) \to (\Y, \X))$ and similarly
$\mathcal{K}_A(\Y, \X) =\mathcal{K}_A((\Y, \X) \to (\Y, \X))$. 
Then, Definitions \ref{GenMorph} and \ref{GenKMorph}
collapse to 
\begin{equation}\label{L_A(X,Y)}
\mathcal{L}_A(\Y, \X) = \{ t \in \Li(L^p(\mu_1)) \colon tx \in \X \mbox{ and } yt \in \Y \mbox{ for all } x \in \X, y \in \Y \}.
\end{equation}
By definition, $\Li_A(\Y, \X)$ is already an $L^p$-operator subalgebra of $\Li(L^p(\mu_1))$. Similarly, 
 \begin{equation}\label{KK_A(X,Y)}
 \mathcal{K}_A(\Y, \X) = \cj{\op{span}\{\theta_{x,y} \colon x \in \X \mbox{ and } y \in \Y \} } \subseteq \Li(L^p(\mu_1)).
 \end{equation}
By definition $\mathcal{K}_A(\Y, \X)  \subseteq \mathcal{L}_A(\Y, \X)$.
\begin{tpr}\label{K<L}
$\mathcal{K}_A(\Y, \X)$ is a closed two-sided ideal in $\mathcal{L}_A(\Y, \X)$. 
\end{tpr}
\begin{proof}
By construction, $\mathcal{K}_A(\Y, \X)$  is a closed subset of $\mathcal{L}_A(\Y, \X)$. 
The ideal claim follows at once from Lemma \ref{K_ID}.
\end{proof}

Below we will compute $\mathcal{L}_A(\Y, \X)$ and $\mathcal{K}_A(\Y, \X)$ 
for some of our known examples.

\begin{ex}\label{A_Amorph}
Let $A$ be an $L^p$-operator algebra with a c.a.i. and let
$(A, A)$ be the $L^p$-module over $A$ from Example \ref{A_A}. 
Then
\[
\mathcal{K}_A(A,A) \cong A. 
\]
Indeed, the Cohen-Hewitt factorization theorem 
(in fact, we only need Theorem 1 in \cite{Cohen1959}) 
implies at once that the map $\theta_{a,b} \mapsto ab$ induces 
an isometric isomorphism from $\mathcal{K}_A(A,A)$ to $A$. 
Next, assuming in addition that $A$ sits nondegenerately 
in $\Li(L^p(\mu))$ (i.e., $AL^p(\mu)$ is a dense subset of $L^p(\mu)$), 
we get 
\[
\Li_A(A,A)  \cong M(A),
\]
where $M(A)$ is the multiplier algebra of $A$ defined as double centralizers. Indeed, Equation \eqref{L_A(X,Y)}  gives
\[
\Li_A(A,A) = \{ t \in \Li(L^p(\mu)) \colon ta \in A, at \in A \mbox{ for all } a \in A\},
\]
which coincides with $M(A)$ thanks to Corollary 3.5 in \cite{BDW2024}. 
\end{ex}

\begin{ex}\label{LpLqMODmorph}
Let $(\Omega, \mathfrak{M}, \mu)$ be a measure space, let $p \in (1, \infty)$, and let $q$ be its Hölder conjugate. We can also include $p=1$ whenever $\mu$ is semi-finite. 
Consider $(L^q(\mu), L^p(\mu))$, the C*-like $L^p$-module over $\C=\Li(\ell_1^p)$
presented in Example \ref{LpLqMOD}. 
Then we claim 
\[
\Li_\C(L^q(\mu), L^p(\mu)) = \Li(L^p(\mu)).
\] 
According to Equation \eqref{L_A(X,Y)}, 
$\Li_\C(L^q(\mu), L^p(\mu))$ is given by 
\[
\{ t \in \Li(L^p(\mu)) \colon t\xi \in L^p(\mu)\text{ for all } \xi \in L^p(\mu) \ \mbox{ and } \  \eta t \in L^q(\mu) \text{ for all } \eta \in L^q(\mu)\},
\]
and therefore $\Li_\C(L^q(\mu), L^p(\mu)) \subseteq  \Li(L^p(\mu))$. We only need to establish the reverse inclusion. Take any $t \in \Li(L^p(\mu))$ and let $t' \in \Li(L^q(\mu))$ be its Banach dual map.
The composition $t\xi$ agrees with $t(\xi)\in L^p(\mu)$ for any $\xi \in L^p(\mu)$ and a direct check shows that $\eta t$ agrees with  $t'(\eta) \in L^q(\mu)$ for any $\eta \in L^q(\mu)$, proving the desired reverse inclusion.
We now claim that
\[
\mathcal{K}_\C(L^q(\mu), L^p(\mu)) = \mathcal{K}(L^p(\mu)).
\]
Indeed, since $L^p(\mu)$ has 
the the approximation property (see Example 4.5 in \cite{Ryan2002}), then
$\mathcal{K}(L^p(\mu))$ is the closure of the finite rank operators. Any 
rank one operator on $L^p(\mu)$ is given by a pair $(\xi, \eta) \in  L^p(\mu)\times L^q(\mu)$
via $\xi_0 \mapsto \xi \langle \eta, \xi_0 \rangle = \theta_{\xi, \eta}\xi_0$. Thus, 
\[
\mathcal{K}_\C(L^q(\mu), L^p(\mu)) = \cj{\op{span}\{ \theta_{\xi ,\eta} \colon \xi \in L^p(\mu), \eta \in L^q(\mu)\}} = \mathcal{K}(L^p(\mu)),
\]
as claimed. 
\end{ex}

The symmetry between Example \ref{LpLqMOD} and Example \ref{LqLpMod}
is actually a particular case of the following result. 

\begin{tpr}
Let $p \in [1, \infty)$, let $A \subseteq \Li(L^p(\mu_0))$ be 
an $L^p$-operator algebra, 
and let $(\Y, \X)$ be an $L^p$-module over $A$ with 
$\X \subseteq \Li( L^p(\mu_0), L^p(\mu_1))$. 
Then $(\X, \Y)$ is an $L^p$-module over $\mathcal{K}_A(\Y, \X) \subseteq \Li(L^p(\mu_1))$. 
\end{tpr}
\begin{proof}
We only need to verify conditions \eqref{lpm5}-\eqref{lpm5.2} in Definition \ref{mainD}.
For any $t \in \mathcal{K}_A(\Y, \X) $ we have $yt \in \Y$ for 
any $y \in \Y$, and $tx \in \X$ for any $x \in \X$.  
This proves both condition \eqref{lpm5} and \eqref{lpm5.1}. 
Finally, since $xy = \theta_{x,y} \in  \mathcal{K}_A(\Y, \X) $, condition \eqref{lpm5.2} holds and we are done. 
\end{proof}

Next we compute the morphisms 
for the standard $L^p$-module of 
an $L^p$-operator algebra, as given in Definition \ref{LPSTDMOD}. 
Two well known facts in the C*-case are $\mathcal{K}_A(\ell^2 \otimes_\C A) = \mathcal{K}(\ell^2) \otimes A$, 
and $\mathcal{L}_A(\ell^2 \otimes_\C A) =M( \mathcal{K}(\ell^2) \otimes A)$.  Our next result shows that, 
provided that we start with reasonably well behaved $L^p$-operator algebra $A$, the same results hold 
for the standard $L^p$-module $(\ell^q, \ell^p)\otimes_p (A, A)$. 

\begin{tpr}\label{Kasp}
Let $p \in [1, \infty)$, let $A \subseteq \Li(L^p(\mu))$ be an $L^p$-operator algebra, and let $\nu$ be counting 
measure on $\Z_{\geq 1}$. If $A$ has a c.a.i., then
\[
\mathcal{K}_A (\ell^q \otimes_p A,\ell^p \otimes_p A) =\mathcal{K}(\ell^p) \otimes_p A \subseteq L^p(\nu \times \mu). 
\]
If in addition $A$ sits nondegenerately in $\Li(L^p(\mu))$, then 
\[
\mathcal{L}_A (\ell^q \otimes_p A,\ell^p \otimes_p A) =M(\mathcal{K}(\ell^p) \otimes_p A)\subseteq L^p(\nu \times \mu). 
\]
\end{tpr}
\begin{proof}
The first claim follows at once from the Cohen-Hewitt factorization Theorem and the fact that 
$\theta_{x \otimes a , y \otimes b}= \theta_{x,y} \otimes ab$ for any $x \in \ell^p$, $y \in \ell^q$, and $a,b \in A$. 
For the second one, since $\mathcal{K}(\ell^p)$ sits nondegenerately in $\Li(\ell^p)$, we have that $\mathcal{K}(\ell^p) \otimes_p A$ 
is nondegenerately represented on $L^p(\nu \times \mu)$ by construction. Further, both $\mathcal{K}(\ell^p) $ and $A$ have c.a.i.'s, and therefore so does $\mathcal{K}(\ell^p) \otimes_p A$. Hence,  Corollary 3.5 in \cite{BDW2024}
gives
\[
M(\mathcal{K}(\ell^p) \otimes_p A) = \{ t \in L^p(\nu \times \mu) \colon tc, ct \in \mathcal{K}(\ell^p) \otimes_p A \text{ for all } c \in  \mathcal{K}(\ell^p) \otimes_p A\}.
\]
On the other hand, by definition, $\mathcal{L}_A (\ell^q \otimes_p A,\ell^p \otimes_p A) $ 
is equal to
\[
 \{ t \in L^p(\nu \times \mu) \colon t(x \otimes a) \in \ell^p \otimes_p A, (y \otimes a)t \in \ell^q \otimes_p A  \text{ for all } x \in \ell^p, y \in \ell^q, a \in A\}.
\]
We will show that both algebras are equal. First notice that, by the first claim, 
\[
\mathcal{K}(\ell^p) \otimes_p A = \cj{\op{span}\{ \theta_{x \otimes a, y \otimes b} \colon x \in \ell^p, y \in \ell^q, a,b \in A\}}.
\] 
Thus, if $t \in \mathcal{L}_A\big( (\ell^q \otimes_p A,\ell^p \otimes_p A) \big)$, 
then $t \theta_{x \otimes a,y \otimes b}=\theta_{t(x\otimes a),y \otimes b} \in \mathcal{K}(\ell^p) \otimes_p A $ and $\theta_{x \otimes b, y \otimes b}t=\theta_{x \otimes b,(y \otimes b)t} \in \mathcal{K}(\ell^p) \otimes_p A $. 
This proves  
\[
\mathcal{L}_A(\ell^q \otimes_p A,\ell^p \otimes_p A) \subseteq M(\mathcal{K}(\ell^p) \otimes_p A).
\]
For the reverse inclusion, let $t \in M(\mathcal{K}(\ell^p) \otimes_p A)$, 
take any $x \otimes a \in \ell^p \otimes_p A$ and use the Cohen-Hewitt factorization Theorem to write $a=a_0a_1$. Then 
\[
t(\theta_{x, \delta_1} \otimes a_0) \in \mathcal{K}(\ell^p) \otimes_p A=\mathcal{K}_A (\ell^q \otimes_p A,\ell^p \otimes_p A),
\]
which implies that 
$t(x \otimes a)=t(\theta_{x, \delta_1} \otimes a_0)(\delta_1 \otimes a_1) \in \ell^p \otimes_p A$. 
A symmetric argument shows that $(y\otimes a)t \in \ell^q \otimes_p A$
for any $y \otimes a\in \ell^q \otimes_p A$, so it follows that $t \in  \mathcal{L}_A (\ell^q \otimes_p A,\ell^p \otimes_p A) $, 
finishing the proof.
\end{proof}
\begin{cor}
Let $p \in [1, \infty)$, let $q$ be its Hölder conjugate, 
and let $A$ be an $L^p$-operator algebra with a bicontractive 
approximate identity. The quotient algebra  $\mathcal{L}_A (\ell^q \otimes_p A,\ell^p \otimes_p A)  / \mathcal{K}_A (\ell^q \otimes_p A,\ell^p \otimes_p A)$ is also an $L^p$-operator algebra. 
\end{cor}
\begin{proof}
The natural c.a.i.\ of $\mathcal{K}_A(\ell^p)$ is in fact bicontractive. Thus, in this case, Proposition \ref{Kasp}  implies that $\mathcal{K}_A(\ell^q \otimes_p A,\ell^p \otimes_p A)=\mathcal{K}(\ell^p) \otimes_p A$ also has a bicontractive approximate identity obtained simply by tensoring the one for $\mathcal{K}(\ell^p)$ with the one for $A$. The result now follows at once from Part (1) of Lemma 4.5 in  \cite{blph2020} . 
\end{proof}

\begin{quest}
Proposition \ref{Kasp} is a $p$-version of a particular instance of Kasparov's theorem  (Theorem 15.2.12 in \cite{wegge-olsen_2004}). A natural question to ask is whether Kasparov's theorem  holds for $L^p$-modules $(\Y,\X)$ over $A$, that is, do we have
\[
\Li_A (\Y, \X)  = M\big( \mathcal{K}_A(\Y, \X) \big)?
\]
\end{quest}

\section{$L^p$-correspondences}\label{s5}

In this section we define the extra structure needed on $L^p$-modules 
to obtain $L^p$-correspondences. We then present an interior tensor product 
construction for these correspondences.

\subsection{$L^p$-correspondences }

Having defined $L^p$-module morphisms in the previous section,
we are now ready to give a definition for correspondences over $L^p$-operator algebras. 

 \begin{defi}\label{mainLpcorrresD}
Let $(\Omega_0, \mathfrak{M}_0, \mu_0), (\Omega_1, \mathfrak{M}_1, \mu_1)$ be measure spaces, let $p \in [1, \infty)$,
let $A$ be an $L^p$-operator algebra, and let $B \subseteq \Li(L^p( \mu_0))$
 be a concrete $L^p$-operator algebra. An \emph{$(A,B)$ $L^p$-correspondence} is a pair $((\Y, \X), \varphi)$
 where $(\Y, \X)$ is an $L^p$-module over $B$ with $\X \subseteq \Li(L^p(\mu_0), L^p(\mu_1))$ and $ \varphi \colon A \to \Li_{B}(\X, \Y)$
 is a contractive homomorphism. When $A=B$ we say that $((\Y, \X), \varphi)$ is an $L^p$-correspondence over $A$.
 \end{defi}
 
We now look back at our examples of $L^p$-modules and make them into 
$L^p$-correspondences. 

\begin{ex}\label{A_ALp}
Let $(A, A)$ be the $L^p$-module from Example \ref{A_A}. 
Let $\varphi_A$ be a contractive automorphism of $A$. Notice that for any 
$a, b \in A$, $\varphi_A(a)b \in A$ and $b\varphi_A(a) \in A$. Therefore,
$\varphi_A(a) \in \Li_A(A,A)$ for all $a \in A$. 
Thus, $((A,A), \varphi_A)$ can be regarded as an $L^p$-correspondence over $A$.
\end{ex}

\begin{ex}\label{LpLqLP}
Let $(\Omega, \mathfrak{M}, \mu)$ be a measure 
space, let $p \in (1, \infty)$ with Hölder conjugate $q$, 
and let $(L^q(\mu), L^p(\mu))$
be the C*-like $L^p$-module over $\C$ from Example \ref{LpLqMOD}.
For each $z \in \C$, we define $\varphi_\C(z) \colon L^p(\mu) \to L^p(\mu)$
by $\varphi_\C(z) = z \cdot \op{id}_{L^p(\mu)}$. 
Then it is clear that $\varphi_\C(z)\xi \in \Li(\ell_1^p,L^p(\mu))$ and $\eta\varphi_\C(z) \in \Li(L^p(\mu), \ell_1^p)$ for all 
$z \in \C$, $\xi \in \Li(\ell_1^p,L^p(\mu))$, and $\eta \in \Li(L^p(\mu), \ell_1^p)$. Hence, $\varphi_\C(z) \in \Li_\C(L^q(\mu), L^p(\mu))$. Finally, 
since $\|\varphi_\C(z)\|=|z|$,
it follows that $((L^q(\mu), L^p(\mu)), \varphi_\C)$ is an $L^p$-correspondence over $\C$.
\end{ex}

\begin{ex}\label{lplqLp}
Let $p \in (1, \infty)$, let 
let $p \in (1, \infty)$ with Hölder conjugate $q$,
and $(\ell^q_d, \ell^p_d)$ be the C*-like $L^p$-module from Example \ref{lplqMOD}. For each 
$z \in \C$ let $\varphi_d(z): \ell_d^p \to \ell_d^p$ be given by 
\[
\varphi_d(z)(\zeta(1), \ldots, \zeta(d)) = (z\zeta(1), \ldots, z\zeta(d))
\] 
Then this is a particular example of Example \ref{LpLqLP}, so
it follows that $((\ell_d^q,\ell_d^p), \varphi_d)$ is an $L^p$-correspondence over $\C$.
\end{ex}

\begin{ex}
Let $(\ell_d^q \otimes_p A, \ell_d^p \otimes_p A)$ be the $L^p$-module from Example \ref{AdAdMOD}. 
For each 
$a \in A$ let $\varphi(a): L^p(\Omega, \mu)^d \to  L^p(\Omega, \mu)^d $ be given by 
\[
\varphi(a)(\xi_1, \ldots, \xi_d) = (a\xi_1, \ldots, a\xi_d)
\] 
Then it is clear that $\varphi(a)x \in \ell_d^p \otimes_p A$ and $y\varphi(a) \in \ell_d^q \otimes_p A$ for all 
$x \in \ell_d^p \otimes_p A$ and $y \in \ell_d^q \otimes_p A$. 
Since $\|\varphi(a)\| \leq \|a\|$,
it follows that $((\ell_d^q \otimes_p A, \ell_d^p \otimes_p A), \varphi)$ is an $L^p$-correspondence over $A$.
\end{ex}

\subsection{Tensor Product of $L^p$-correspondences}
 
Before giving our main definition for the interior tensor product of $L^p$-correspondences, we briefly recall
the setting for the C*-case guaranteed by Theorem 4.4 in \cite{Delfin_2022}.
If $(\X, \varphi_\X)$ is an $(A, B)$ C*-correspondence represented by 
$(\pi_A, \pi_B, \pi_\X)$ on $(\Hi_1, \Hi_2)$, and
$(\Y, \varphi_\Y)$ is a $(B, C)$ C*-correspondence 
represented by $(\pi_B, \pi_C, \pi_\Y)$ on $(\Hi_0, \Hi_1)$. 
Then given some nondegeneracy conditions, $(\X\otimes_{\varphi_\Y}\Y, \widetilde{\varphi_\X})$
can be represented on $(\Hi_0,\Hi_2)$ via the map defined by $x \otimes y \mapsto \pi_X(x)\pi_\Y(y)$. 
Furthermore, in this scenario, if $\kappa_C$ is the isomorphism from 
$\Li_{\pi_C(Y)}(\pi_\Y(\Y))$ 
to  $\Li_{C}(\Y)$ given by Part (ii) in Proposition 3.10 in \cite{Delfin_2022}, then 
it is not hard to check that
\[
\varphi_\Y(\langle x_1, x_2 \rangle_B) = \kappa_C(\pi_X(x_1)^*\pi_X(x_2)).
\]
This essentially means that, at the concrete level, 
the left action $\varphi_\Y$ acts as the identity on $\langle \X, \X \rangle_B$. 
Translating all this to the $L^p$-case gives rise to the following definition.

\begin{defi}\label{TensorLpCorres}
Let $p \in [1, \infty)$ and for each $j=0,1,2$ let $(\Omega_j, \mathfrak{M}_j, \mu_j)$ be a measure space.
Set $E_j=L^p(\mu_j)$ for $j=0,1,2$ and 
let $A$ be an $L^p$-operator algebra, and let
$B \subseteq \Li(E_1)$ and
 $C \subseteq \Li(E_0)$ be concrete $L^p$-operator algebras. 
Suppose $((\Y, \X), \varphi)$ is an $(A,B)$ $L^p$-correspondence with $\X \subseteq \Li(E_1,E_2)$ and $\Y \subseteq  \Li(E_2,E_1)$. 
Suppose also that $((\W, \V), \rho)$ is a $(B,C)$ $L^p$-correspondence with 
$\V \subseteq \Li(E_0,E_1)$, $\W \subseteq \Li(E_1,E_0)$, 
and such that $\rho(\lpp{y}{x}{B})=yx$ for all $x \in \X$ and $y \in \Y$.
Then we define an $(A,C)$-$L^p$-correspondence 
\[
((\Y, \X), \varphi) \otimes_\rho ((\W, \V), \rho) = ((\cj{\W \Y} , \cj{\X\V} ), \widetilde{\varphi})
\]
where $\widetilde{\varphi}: A \to \Li_C(\cj{\W \Y} , \cj{\X\V})$ is determined by
\[
\widetilde{\varphi}(a)\xi = \varphi(a)\xi,
\]
for any $\xi \in E_2$.
\end{defi}

We end the paper by checking that the objects defined in Definition \ref{TensorLpCorres} form indeed an $(A,C)$ $L^p$-correspondence. 
We first check that $(\cj{\W \Y} , \cj{\X\V} )$ is indeed an $L^p$-module 
over $C$. By Definition $\cj{\X\V}$ and $\cj{\W \Y}$ are closed subspaces of 
bounded operators of $\Li(E_2 , E_0)$ and $\Li(E_0 , E_2)$. We now check all the 
conditions in Definition \ref{mainD}. 
Let $x \in \X$, $v \in \V$ and $c \in C$. Then we know that $vc \in \V$ 
and therefore $x(vc) \in \X\V$. This is enough to see that $\cj{\X\V} C \subseteq \cj{\X\V}$, 
giving Condition \eqref{lpm5}. Similarly, if $c \in C$, $y \in \Y$ and $w \in \W$ 
we get $c(wy)=(cw)y \in \W\Y$, from where Condition \eqref{lpm5.1} follows. For Condition \eqref{lpm5.2}, take $x \in \X$, $v \in \V$, $y \in \Y$ and $w \in \W$. 
Then since $yx \in B$ satisfies $\rho(\lpp{y}{x}{B})=yx$, it follows that 
\[
\lpp{wy}{xv}{C} = (wy)(xv)=w\rho(\lpp{y}{x}{B})v \in \W\V \subseteq C,
\]
because $\rho(b)v \in \V$ for any $b \in B$. It remains to check
that $\widetilde{\varphi}(a) \in \Li_C(\cj{\W\Y}, \cj{\X\V}))$ for any $a \in A$. Indeed, 
it is clear that for any $x \in \X$ and $v \in \V$
\[
\widetilde{\varphi}(a)xv=(\varphi(a)x)v \in \X\V,
\]
and also that for each $y \in \Y$ and $w \in \W$
\[
(wy)\widetilde{\varphi}(a)=w(y\varphi(a)) \in \W\Y.
\]
Finally, since $\| \widetilde{\varphi}(a)\| = \| \varphi(a)\|$, it now follows that 
 $\widetilde{\varphi}(a) \in \Li_C(\X \otimes_B \V, \Y \otimes_B\W)$. 
 Therefore, the ingredients in Definition \ref{TensorLpCorres} do give rise to 
an  $(A,C)$ $L^p$-correspondence. 

\bibliographystyle{plain}
\bibliography{Lp-mod.bib}
\end{document}